\documentclass[11pt, a4paper]{article}
\usepackage{amsmath,amsthm,amsfonts,amssymb, url}
\usepackage{latexsym}
\usepackage{graphicx}
\usepackage{enumerate}

\newtheorem{theorem}{Theorem}
\newtheorem{lemma}[theorem]{Lemma}
\newtheorem{corollary}[theorem]{Corollary}

\newcommand{\lab}[1]{\label{#1}}     

\newcommand{\be}{\begin{equation}}
\newcommand{\ee}{\end{equation}}
\newcommand{\bea}{\begin{eqnarray}}
\newcommand{\eea}{\end{eqnarray}}
\newcommand{\bean}{\begin{eqnarray*}}
\newcommand{\eean}{\end{eqnarray*}}

\long\def\delete#1{}

\def\AA{{\cal A}}
\def\BB{{\cal B}}
\def\CC{{\cal C}}

\def\EE{{\cal E}}
\def\FF{{\cal F}}
\def\WW{{\cal C}}

\title{{\bf Hadwiger's conjecture for  the complements of Kneser graphs\thanks{Research supported by ARC Discovery Project DP120101081.}}}
\author{Guangjun Xu$^{a, b}$ and  Sanming Zhou$^b$  \\ \\
{\small 
 $^a$School of Mathematics and Computing Science} \\
 {\small  Zunyi Normal College, Zunyi 563002, China }  \\
 {\small   {\texttt {gjxu11@gmail.com} }} \smallskip \\ 
{\small
$^b$School of Mathematics and Statistics}\\
{\small The University of Melbourne,  Parkville, VIC 3010,   Australia}\\
{\small   {\texttt{smzhou@ms.unimelb.edu.au}}}
}

\date{}

\begin{document}

\openup 0.5\jot\maketitle

\vspace{-1cm}

\smallskip
\begin{abstract}

Hadwiger's conjecture asserts that every graph with chromatic number $t$ contains a complete minor of order $t$. Given integers $n \ge 2k+1 \ge 5$, the Kneser graph $K(n, k)$ is the graph with vertices the $k$-subsets of an $n$-set such that two vertices are adjacent if and only if the corresponding $k$-subsets are disjoint. 
We prove that Hadwiger's conjecture is true for the complements of Kneser graphs. 

 \medskip


\end{abstract}

\section{Introduction}

A graph  $H$
is a {\em minor} of a graph  $G$ if  a graph isomorphic to  $H$ can be obtained from  a subgraph of   $G$ by contracting
edges. An {\em $H$-minor} is a minor isomorphic to $H$. 
The {\em Hadwiger number} of   $G$, denoted by $h(G)$, is the  maximum integer $t$   such that
  $G$ contains   a   $K_t$-minor, where $K_t$ is the complete graph with $t$ vertices. 
  
Hadwiger \cite{Had43} conjectured that every graph that is not $(t-1)$-colourable contains a $K_t$-minor; that is, $h(G)\ge \chi(G)$ for every graph $G$, where $\chi(G)$ is the chromatic number of $G$. Hadwiger's conjecture is widely believed to be one of the most difficult and beautiful problems in graph theory. It has been proved \cite{RST93} for  graphs with  $\chi(G)\le 6$, and is open for graphs
with   $\chi(G)\ge 7$. It has also been proved for certain special classes of graphs, including powers of cycles and their complements \cite{ll07}, proper circular arc graphs \cite{bc09}, line graphs \cite{rs04}, quasi-line graphs \cite{co08} and $3$-arc graphs \cite{wxz}. See \cite{toft96} for a survey. 

A strengthening of Hadwiger's conjecture due to Haj\'os asserts that every graph $G$ with $\chi(G)\ge t$ contains a subdivision of $K_{t}$. Catlin \cite{catl} proved that Haj\'os' conjecture fails for every $t \ge 7$. Obviously, if Hadwiger's conjecture is false, then counterexamples must be found 
among counterexamples to Haj\'os' conjecture. In \cite{thom} Thomassen presented several new classes of counterexamples to Haj\'os' conjecture, including the complements of the Kneser graphs $K(3k-1, k)$ for sufficiently large $k$. (The {\em Kneser graph} $K(n, k)$ is the graph with vertices the $k$-subsets of an $n$-set such that two vertices are adjacent if and only if the corresponding $k$-subsets are disjoint.) He wrote \cite{thom} that `it does not seem obvious' that these classes all satisfy Hadwiger's conjecture. Motivated by this comment, we prove in this paper that indeed the complement of every Kneser graph satisfies Hadwiger's conjecture. We notice that in the special case when $k$ divides $n$ this was established in \cite{ll07}.

Throughout the paper we use $\overline{K}(n,k)$ (instead of $\overline{K(n,k)}$) to denote the complement of $K(n,k)$. The main result of this paper is as follows.
   
\begin{theorem}
\label{th1}
Let $n$ and $k$ be integers with $n \ge 2k+1 \ge 5$. 
The complement $\overline{K}(n,k)$ of the Kneser graph $K(n,k)$ satisfies Hadwiger's conjecture; that is,
$$
h(\overline{K}(n,k))\geq \chi (\overline{K}(n,k)).
$$  
\end{theorem}

In the case when $2k+1 \le n \le 3k-1$, the independence number of $\overline{K}(n,k)$ is equal to $2$, and so Theorem \ref{th1} asserts that Hadwiger's conjecture is true for this special family of graphs with independence number 2. Moreover, in this case the gap between the Hadwiger number and the chromatic number can be arbitrarily large when $n, k$ vary (see the proof of Corollary \ref{cor1}). In general, Hadwiger's conjecture for graphs of independence number 2 is an interesting but challenging problem; see a related discussion in \cite{co10}. 

Since $\overline{K}(n,2)$ is the line graph of $K_n$ and Hadwiger's conjecture is true for all line graphs \cite{rs04}, the result in Theorem \ref{th1} is known when $k=2$. In the rest of the paper we prove Theorem \ref{th1} for $k \ge 3$.

\section{Preliminaries}

We always use $n$ and $k$ to denote positive integers with $n \ge 2k+1 \ge 7$. 
Denote $[n]=\{1,2,\ldots,n\}$ and call its elements {\em labels}. Denote $[i, j] = \{i, i+1, \ldots, j\}$ for integers $i \le j$. Denote the set of all $k$-subsets of $[n]$ by ${[n] \choose k}$. 
We take the Kneser graph $K(n,k)$ as defined on the vertex set ${[n] \choose k}$ such that two members of ${[n] \choose k}$ are adjacent if and only if they are disjoint. We will use the following well-known result in the proof of Theorem \ref{th1}.
   
\begin{lemma}(Baranyai \cite{bar73})
\label{le1} 
$\chi (\overline{K}(n,k)) = \left\lceil N/\lfloor \frac{n}{k}\rfloor \right\rceil$, where $N= {n \choose k}$.
\end{lemma}

The {\em complete $k$-uniform hypergraph} $K_n^k$ is the hypergraph with $n$ vertices and all possible hyperedges of size $k$. We take $K_n^k$ to have vertex set $[n]$ and hyperedge set ${[n] \choose k}$; in this way each $k$-subset of $[n]$ is viewed as a vertex of $\overline{K}(n,k)$ as well as a hyperedge of $K_n^k$. 
A uniform hypergraph is called {\em almost regular} if the degrees of any two vertices differ by at most one, where the degree of a vertex is the number of hyperedges containing the vertex. We treat a family of hyperedges of a hypergraph as a spanning sub-hypergraph with the same vertex set as the hypergraph under consideration.

\begin{lemma}(Baranyai \cite{bar73})
\label{le2a}
 Let $a_1, a_2, \ldots, a_l$ be positive integers such that $\Sigma_{i=1}^l a_i = {n \choose k}$. 
 Then the set of hyperedges of 
 $K_n^k$ can be partitioned into $\EE_1, \EE_2, \ldots, \EE_l$ such that for $1\le j\le l$, $|\EE_j|=a_j$ and $\EE_j$ is an almost regular hypergraph (with the same vertex set as $K_n^k$). 
\end{lemma}     
 
Denote by $\AA_i(n,k)$, $1 \le i \le n-k+1$, the family of $k$-subsets of $[n]$ with $i$ as the smallest label. That is, $\AA_i(n,k) = \{A \in {[n] \choose k}: i \in A,\, A \setminus \{i\} \subseteq [i+1, n]\}$. It is clear that $|\AA_i(n,k)| = {n-i \choose k-1},\;\, 1 \le i \le n-k+1$. We say that a label of $[n]$ is {\em covered} by a family $\FF \subseteq {[n] \choose k}$ if it is in at least one member of $\FF$.

\begin{lemma}
\label{le3}
Let $i$ be an integer between $1$ and $n-k+1$, and $l$ an integer between $1$ and ${n-i \choose k-1}$. Let $d_i = \lfloor {n-i \choose k-1}/l \rfloor$. Then $\AA_i(n,k)$ can be partitioned into $\AA_{i1}^l(n,k), \AA_{i2}^l(n,k), \ldots, \AA_{i{d_i}}^l(n,k)$ each with size $|\AA_{ij}^l(n,k)| = l$, together with $\AA_{i,{d_i}+1}^l(n,k)$ of size ${n-i \choose k-1} - {d_i}l$ when ${n-i \choose k-1}$ is not divisible by $l$, such that for $1\le j \le d_i$ the hyperedges of $\AA_{ij}^l(n,k)$ cover at least $\min\{n-i+1, l(k-1)+1\}$ labels of $[n]$.
\end{lemma}

\begin{proof}  
Since $\AA_i(n,k) = \{X \cup \{i\}: X \in {[i+1, n] \choose k-1}\}$, by applying Lemma \ref{le2a} to $K_{n-i}^{k-1}$ with vertex set $[i+1, n]$ and setting $a_1=a_2= \cdots=a_{d_i}=l$ together with $a_{{d_i}+1}={n-i \choose k-1} - {d_i}l$ if $l$ is not a divisor of ${n-i \choose k-1}$, we obtain that $\AA_i(n,k)$ can be partitioned into $\AA_{i1}^l(n,k), \AA_{i2}^l(n,k), \ldots, \AA_{i d_i}^l(n,k)$ together with $\AA_{i, d_i + 1}^l(n,k)$ if $l$ is not a divisor of ${n-i \choose k-1}$, whose sizes are as stated in the lemma such that $\BB_{ij}^l(n,k) = \{A \setminus \{i\}: A \in \AA_{ij}^l(n,k)\}$, $1\le j \le d_i$ is an almost regular hypergraph with vertex set $[i+1, n]$. Hence, if $n-i > l(k-1)$, then for $1\le j \le d_i$ each vertex $v \in [i+1, n]$ has degree $0$ or $1$ in $\BB_{ij}^l(n,k)$, and so the hyperedges of $\AA_{ij}^l(n,k)$ cover  $l(k-1)+1$ labels of $[i, n]$. If $n-i \le l(k-1)$, then for $1\le j \le d_i$ each vertex $v \in [i+1, n]$ has positive degree in $\BB_{ij}^l(n,k)$, for otherwise all labels of $[i+1, n]$ would have degrees $0$ or $1$ in $\BB_{ij}^l(n,k)$ with $0$ occurring at least once, yielding $n-i > l(k-1)$, a contradiction. Thus, if $n-i \le l(k-1)$, then the hyperedges of $\AA_{ij}^l(n,k)$ cover all labels of $[i, n]$.  
\end{proof}

Denote by $\WW(n,k)$ the family of $k$-subsets of $[n]$ containing $n$. Then $|\WW(n,k)| = {n-1 \choose k-1}$. Similar to the proof of Lemma \ref{le3}, we can prove the following result by using Lemma \ref{le2a}. 
 
\begin{lemma}
\label{le4}
Let $l$ be an integer between $2$ and ${n-1 \choose k-1}$. Let $r = \lfloor {n-1 \choose k-1}/l \rfloor$. Then $\WW(n,k)$ can be partitioned into $\WW_{1}^l(n,k), \WW_2^l(n,k), \ldots, \WW_{r}^l(n,k)$ each with size $|\WW_{i}^l(n,k)| = l$, together with $\WW_{r+1}^l(n,k)$ of size ${n-1 \choose k-1} - rl$ when ${n-1 \choose k-1}$ is not divisible by $l$, such that for $1\le i \le r$ the hyperedges of $\WW_i^l(n,k)$ cover at least $\min\{n, l(k-1)+1\}$ labels of $[n]$.
\end{lemma}    

A $K_t$-minor of a graph $G$  can be viewed as a family of $t$ vertex-disjoint connected subgraphs of $G$ such that there exists at least one edge of $G$ between each pair of subgraphs. 
	  Each subgraph in the family is called a {\em branch set}. 

In the proof of Theorem \ref{th1} we will use the following well known identity: for integers $a \ge b \ge 0$,
$$
\sum_{i=0}^{a}{i \choose b}={a+1 \choose b+1}.
$$

\section{Proof of Theorem \ref{th1}}

Throughout this section we always write $n=sk+t$, where $s\ge 2$ and $0\le t\le k-1$.  
     
 
\subsection{$s=2$}    
        
\begin{lemma}
\label{le6}
   Let $n=2k+t$, where $k \ge 2$ and $1\le t\le k-1$. Then
\begin{equation}
\label{eq:1}   
h(\overline{K}(n,k)) \ge \left\{ 
  \begin{array}{l l}
   \frac{1}{2} {n \choose k} +  \frac{1}{2}{n-1 \choose k-1}- \frac{1}{2}{n-k \choose k}-\frac{k-1}{2},   & \quad \text{$1\le t\le k-2$} \\[0.2cm]
  \frac{1}{2} {n \choose k} + \frac{1}{6}{n-1 \choose k-1} - \frac{1}{2}{n-1-k \choose k}-\frac{k-1}{2}-\frac{2}{3}, & \quad \text{$t=k-1$}.         
  \end{array} \right.
\end{equation}  
\end{lemma}

\begin{proof} 
\textsf{Case 1: $1 \le t \le k-2$.}\,
Let $\AA_{i1}^2(n,k), \AA_{i2}^2(n,k), \ldots, \AA_{i d_i}^2(n,k)$ be as in Lemma \ref{le3} each with size $2$, where $1 \le i \le k$ and $d_i = \lfloor {n-i \choose k-1}/2 \rfloor$. Then by Lemma \ref{le3} the hyperedges of $\AA_{ij}^2(n,k)$ ($1 \le i \le k, 1 \le j \le d_i$) cover at least $\min\{n-i+1, 2k-1\} \ge \min\{n-k+1, 2k-1\}$ labels of $[n]$. 

Since $1 \le t \le k-2$, we have $\min\{n-k+1, 2k-1\}+k\ge n+1$ and hence for $i \ne i'$, $1 \le j \le d_i$, $1 \le j' \le d_{i'}$ there is at least one edge of $\overline{K}(n,k)$ between the subgraphs induced by $\AA_{ij}^2(n,k)$ and $\AA_{i' j'}^2(n,k)$. Similarly, for $2 \le i \le k$ and $1\le j \le d_i$, each $A \in \AA_1(n,k)$ is adjacent to at least one member of $\AA_{ij}^2(n,k)$ in $\overline{K}(n,k)$. Since for $1\le j \le d_i$ all hyperedges in $\AA_{ij}^2(n,k)$ contain $i$, $\cup_{j=1}^{d_i} \AA_{ij}^2(n,k)$ induces a complete subgraph of $\overline{K}(n,k)$. Therefore, the isolated vertices $A \in \AA_1(n,k)$ of $\overline{K}(n,k)$ and the subgraphs of $\overline{K}(n,k)$ induced by $\AA_{ij}^2(n,k)$ for $2 \le i \le k$ and $1 \le j \le d_i$ are branch sets of $\overline{K}(n,k)$, that is, they give rise to a complete minor of $\overline{K}(n,k)$.   
The number of such branch sets is given by
\begin{eqnarray}
\nonumber
|\AA_1(n,k)| + \sum_{i=2}^k d_i  
                  & \ge & {n-1 \choose k-1} + \frac{1}{2} \sum_{i=2}^k {n-i \choose k-1} - \frac{k-1}{2}\\ \nonumber
                    & = & {n-1 \choose k-1} + \frac{1}{2}\left({n-1 \choose k} -{n-k \choose k}\right) -\frac{k-1}{2} \\ \nonumber
                    & = & \frac{1}{2} {n \choose k} +  \frac{1}{2}{n-1 \choose k-1}- \frac{1}{2}{n-k\choose k}-\frac{k-1}{2}.\lab{eq:bs}
\end{eqnarray}	
This proves the first bound in (\ref{eq:1}). 

\medskip
\textsf{Case 2: $t = k-1$.}\,
We now prove the second bound in (\ref{eq:1}). By what we proved in Case 1 with $n$ replaced by $n-1$, we have a complete minor of $\overline{K}(n-1,k)$ of order no less than $\frac{1}{2} {n-1 \choose k} +  \frac{1}{2}{n-2 \choose k-1}- \frac{1}{2}{n-k-1 \choose k}-\frac{k-1}{2}$ such that all vertices involved are members of ${[n-1] \choose k}$. Since $\overline{K}(n-1,k)$ is a subgraph of $\overline{K}(n,k)$, this complete minor is also a minor of $\overline{K}(n,k)$. 
 
Let $\CC_{1}^3(n,k), \CC_2^3(n,k), \ldots, \CC_{r}^3(n,k)$ be as in Lemma \ref{le4} each with size $3$, where $r = \lfloor {n-1 \choose k-1}/3 \rfloor$. Then by Lemma \ref{le4} the hyperedges of $\CC_i^3(n,k)$ ($1\le i \le r$) cover at least $\min\{3k-1, 3(k-1)+1\} = 3k-2$ labels of $[n]$. Since $n=3k-1$ and $k \ge 2$, it follows that there is at least one edge of $\overline{K}(n,k)$ between each $\CC_i^3(n,k)$ ($1 \le i \le r$) and each of the branch sets in the complete minor mentioned in the previous paragraph. These branch sets and the subgraphs induced by $\CC_i^3(n,k)$ ($1 \le i \le r$) form a larger family of branch sets of $\overline{K}(n,k)$, because $\cup_{i=1}^r \CC_i^3(n,k)$ induces a complete subgraph of $\overline{K}(n,k)$ as all members of $\CC(n,k)$ contain $n$. The number of branch sets in this enlarged family is no less than
\begin{eqnarray*}
\nonumber
                  & &  \frac{1}{2} {n-1 \choose k} +  \frac{1}{2}{n-2 \choose k-1}- \frac{1}{2}{n-1-k \choose k}-\frac{k-1}{2}  +  \left\lfloor {n-1 \choose k-1}\Big/3 \right\rfloor \\
                    & \ge &  \frac{1}{2} {n-1 \choose k} +\frac{1}{3}{n-1 \choose k-1} - \frac{2}{3}+ \frac{1}{2}{n-2 \choose k-1}- \frac{1}{2}{n-1-k \choose k}-\frac{k-1}{2} \\
                   & = & \frac{1}{2} {n \choose k} + \left(\frac{2k-1}{6k-4}- \frac{1}{6}\right){n-1 \choose k-1} - \frac{1}{2}{n-1-k \choose k}-\frac{k-1}{2}-\frac{2}{3} \\
                    & \ge & \frac{1}{2} {n \choose k} + \frac{1}{6}{n-1 \choose k-1} - \frac{1}{2}{n-1-k \choose k}-\frac{k-1}{2}-\frac{2}{3}. 
			\end{eqnarray*} 
This proves the second bound in (\ref{eq:1}).
\end{proof}

\begin{corollary}
\label{cor1}
   Let $k\ge 3$  and  $2k+1\le n\le 3k-1$. Then 
    $$h(\overline{K}(n,k)) \ge \left\lceil {n \choose k}\Big/2 \right\rceil =  \chi(\overline{K}(n,k)).$$
\end{corollary}

	\begin{proof}  
Write $n=2k+t$ with $1\le t\le k-1$. One can verify that ${n-1 \choose k-1}- {n-k \choose k}-(k-1) \ge 0$, and that $\frac{1}{6}{n-1 \choose k-1} - \frac{1}{2}{n-1-k \choose k}-\frac{k-1}{2}-\frac{2}{3}\ge 0$ for $n=3k-1$ with $k\ge 3$. Thus $h(\overline{K}(n,k)) \ge  \frac{1}{2} {n \choose k}$ by Lemma \ref{le6}, which implies $h(\overline{K}(n,k)) \ge \left\lceil {n \choose k}\big/2 \right\rceil =  \chi(\overline{K}(n,k))$ by Lemma \ref{le1}. 
\end{proof}

\subsection{$s=3$}

\begin{lemma}
\label{le8}
   Let $n=3k+t$, where $k \ge 3$ and $0\le t\le k-1$. Then
   \[h(\overline{K}(n,k)) \ge \left\{ 
  \begin{array}{l l l l l}
     \frac{1}{3} {n \choose k} +  \frac{2}{3}{n-1 \choose k-1}- \frac{1}{3}{n-k\choose k}-\frac{2(k-2)}{3},  & \quad \text{$0 \le t\le k-3$}\\ [0.2cm]
   \frac{1}{3} {n \choose k} + \frac{1}{3}{n-1 \choose k-1} - \frac{1}{3}{n-k-1 \choose k}-\frac{2(k-2)}{3}-\frac{3}{4},  & \quad \text{$t=k-2$}\\ [0.2cm]
    60,    & \quad \text{$t=k-1=2$} \\ [0.2cm]
				505,    & \quad \text{$t=k-1=3$} \\ [0.2cm]
     \frac{1}{3} {n \choose k} +   \frac{1}{6}{n-1 \choose k-1} +  \frac{1}{6(n-1)}{n-1 \choose k-1} - \frac{1}{3}{n-k-2 \choose k}-\frac{2(k-2)}{3}-\frac{3}{2},  & \quad \text{$t=k-1\ge 4$}              .         
  \end{array} \right.\]  
\end{lemma}

\begin{proof}    
\textsf{Case 1: $0\le t\le k-3$.}\, Let $\AA_{i1}^3(n,k), \AA_{i2}^3(n,k), \ldots, \AA_{id_i}^3(n,k)$ be as in Lemma \ref{le3} each with size $3$, where $1 \le i \le k$ and $d_i = \lfloor {n-i \choose k-1}/3 \rfloor$. Then by Lemma \ref{le3} the hyperedges of $\AA_{ij}^3(n,k)$ ($1 \le i \le k, 1 \le j \le d_i$) cover at least $\min\{n-i+1, 3k-2\} \ge \min\{n-k+1, 3k-2\}$ labels of $[n]$.

Since $0\le t\le k-3$, we have $\min\{n-k+1, 3k-2\}+k\ge n+1$ and hence for $i \ne i'$, $1 \le j \le d_i$, $1 \le j' \le d_{i'}$ there is at least one edge of $\overline{K}(n,k)$ between the subgraphs induced by $\AA_{ij}^3(n,k)$ and $\AA_{i' j'}^3(n,k)$. Similarly, for $2 \le i \le k$ and $1\le j \le d_i$, each $A \in \AA_1(n,k)$ is adjacent to at least one member of $\AA_{ij}^3(n,k)$ in $\overline{K}(n,k)$. Since for $1\le j \le d_i$ all hyperedges in $\AA_{ij}^3(n,k)$ contain $i$, $\cup_{j=1}^{d_i} \AA_{ij}^3(n,k)$ induces a complete subgraph of $\overline{K}(n,k)$. Therefore, the isolated vertices $A \in \AA_1(n,k)$ of $\overline{K}(n,k)$ and the subgraphs of $\overline{K}(n,k)$ induced by $\AA_{ij}^3(n,k)$ for $2 \le i \le k$ and $1 \le j \le d_i$ are branch sets of $\overline{K}(n,k)$ yielding a complete minor. The number of such branch sets is given by (noting that $3$ divides $|\AA_{t+3}(n,k)| = {3(k-1) \choose k-1}$)
\begin{eqnarray}
\lab{eq:bs1}
|\AA_1(n,k)| + \sum_{i=2}^k d_i 
& = & {n-1 \choose k-1}+\sum_{i=2}^k \left\lfloor {n-i \choose k-1}/3 \right\rfloor\\ \nonumber 
& \ge & {n-1 \choose k-1}+\frac{1}{3}\sum_{i=2}^k {n-i \choose k-1}  - \frac{2(k-2)}{3} \\ \nonumber
& = & {n-1 \choose k-1}+\frac{1}{3}\left({n-1 \choose k} -{n-k \choose k}\right) -\frac{2(k-2)}{3} \\ \nonumber
& =  & \frac{1}{3} {n \choose k} +  \frac{2}{3}{n-1 \choose k-1}- \frac{1}{3}{n-k\choose k}-\frac{2(k-2)}{3}. 
\end{eqnarray}

\medskip
\textsf{Case 2: $t=k-2$.}\, By what we proved in Case 1 with $n$ replaced by $n-1$, we have a complete minor of $\overline{K}(n-1,k)$ (and hence of $\overline{K}(n,k)$) with order no less than $\frac{1}{3} {n-1 \choose k} +  \frac{2}{3}{n-2 \choose k-1}- \frac{1}{3}{n-k-1\choose k}-\frac{2(k-2)}{3}$ such that all vertices involved are members of ${[n-1] \choose k}$. 

Let $\CC_{1}^4(n,k), \CC_2^4(n,k), \ldots, \CC_{r}^4(n,k)$ be as in Lemma \ref{le4} each with size $4$, where $r = \lfloor {n-1 \choose k-1}/4 \rfloor$. Then by Lemma \ref{le4} the hyperedges of $\CC_i^4(n,k)$ ($1\le i \le r$) cover at least $\min\{4k-2, 4(k-1)+1\} = 4k-3$ labels of $[n]$. Since $n=4k-2$ and $k \ge 3$, it follows that there is at least one edge of $\overline{K}(n,k)$ between each $\CC_i^4(n,k)$ ($1 \le i \le r$) and each of the branch sets in the above-mentioned complete minor. These branch sets and the subgraphs induced by $\CC_i^4(n,k)$ ($1 \le i \le r$) form a larger family of branch sets of $\overline{K}(n,k)$, because $\cup_{i=1}^r \CC_i^4(n,k)$ induces a complete subgraph of $\overline{K}(n,k)$ as all members of $\CC(n,k)$ contain $n$. The number of branch sets in this enlarged family is no less than
\begin{eqnarray*}
\nonumber
   &   & \frac{1}{3} {n-1 \choose k} +  \frac{2}{3}{n-2 \choose k-1}- \frac{1}{3}{n-k-1\choose k}-\frac{2(k-2)}{3} + \left\lfloor \frac{1}{4}{n-1 \choose k-1}\right\rfloor \\ \nonumber
& \ge &  \frac{1}{3} {n-1 \choose k} +\frac{1}{4}{n-1 \choose k-1} -\frac{3}{4}+ \frac{2}{3}{n-2 \choose k-1}- \frac{1}{3}{n-1-k \choose k}-\frac{2(k-2)}{3} \\ 
& = & \frac{1}{3} {n \choose k} + \left(\frac{2(n-k)}{3n-3}- \frac{1}{12}\right){n-1 \choose k-1} - \frac{1}{3}{n-1-k \choose k}-\frac{2(k-2)}{3}-\frac{3}{4} \\ \nonumber
& \ge & \frac{1}{3} {n \choose k} + \frac{1}{3}{n-1 \choose k-1} - \frac{1}{3}{n-k-1 \choose k}-\frac{2(k-2)}{3}-\frac{3}{4}. 
\end{eqnarray*}

\medskip
\textsf{Case 3: $t=k-1$.}\, Replacing $n$ by $n-1$ in Case 2 above, we obtain a complete minor of $\overline{K}(n,k)$ of order no less than $\frac{1}{3} {n-1 \choose k} + \frac{1}{3}{n-2 \choose k-1} - \frac{1}{3}{n-k-2 \choose k}-\frac{2(k-2)}{3}-\frac{3}{4}$ such that all vertices involved are members of ${[n-1] \choose k}$. 

Let $\CC_{1}^4(n,k), \CC_2^4(n,k), \ldots, \CC_{r}^4(n,k)$ be as in Lemma \ref{le4} each with size $4$, where $r = \lfloor {n-1 \choose k-1}/4 \rfloor$. Then by Lemma \ref{le4} the hyperedges of $\CC_i^4(n,k)$ ($1\le i \le r$) cover at least $\min\{4k-1, 4(k-1)+1\} = 4k-3$ labels of $[n]$. Since $n=4k-1$ and $k \ge 3$, it follows that there is at least one edge of $\overline{K}(n,k)$ between each $\CC_i^4(n,k)$ ($1 \le i \le r$) and each of the branch sets in the complete minor mentioned in the previous paragraph. These branch sets and the subgraphs induced by $\CC_i^4(n,k)$ ($1 \le i \le r$) form a larger family of branch sets of $\overline{K}(n,k)$, because $\cup_{i=1}^r \CC_i^4(n,k)$ induces a complete subgraph of $\overline{K}(n,k)$ as all members of $\CC(n,k)$ contain $n$. The number of branch sets in this enlarged family is no less than
\begin{eqnarray*}
\nonumber
&   & \frac{1}{3} {n-1 \choose k} + \frac{1}{3}{n-2 \choose k-1} - \frac{1}{3}{n-k-2 \choose k}-\frac{2(k-2)}{3}-\frac{3}{4}+\left\lfloor \frac{1}{4}{n-1 \choose k-1}\right\rfloor \\ \nonumber
& \ge & \frac{1}{3} {n-1 \choose k} + \frac{1}{3}{n-2 \choose k-1} - \frac{1}{3}{n-k-2 \choose k}-\frac{2(k-2)}{3}-\frac{3}{2}+\frac{1}{4}{n-1 \choose k-1}\\ \nonumber
& = & \frac{1}{3} {n \choose k} +   \frac{1}{6}{n-1 \choose k-1} +  \frac{1}{6(n-1)}{n-1 \choose k-1} - \frac{1}{3}{n-k-2 \choose k}-\frac{2(k-2)}{3}-\frac{3}{2}.
\end{eqnarray*}  
			
In the case when $t=k-1=2$ or $3$, the lower bound above can be improved. For example, when $t=k-1=2$, by following the argument above but improving the estimate in (\ref{eq:bs1}) we obtain that $\overline{K}(11, 3)$ has a complete minor of order at least
${8 \choose 2} + \left(\left\lfloor {7 \choose 2}/3\right\rfloor + \left\lfloor {6\choose 2}/3\right\rfloor\right) + \left\lfloor {9\choose 2}/4\right\rfloor + \left\lfloor {10\choose 2}/4 \right\rfloor = 28 + 7  +5  + 9  + 11 =  60$.
Similarly, when $t=k-1=3$ we see that $\overline{K}(15, 4)$ has a complete minor of order at least $505$.
\end{proof}

\begin{corollary}
\label{cor2}
   Let $k\ge 3$  and  $3k\le n\le 4k-1$. Then
    $$
    h(\overline{K}(n,k)) \ge \left\lceil {n \choose k}\Big/3 \right\rceil =  \chi(\overline{K}(n,k)).
    $$
\end{corollary}

	\begin{proof}  
Write $n=3k+t$, where  $0\le t\le k-1$. By Lemma \ref{le1} it suffices to prove that $h(\overline{K}(n,k)) \ge \frac{1}{3} {n \choose k}$.
 By Lemma \ref{le8}, when $0 \le t\le k-3$ it suffices to prove $2{n-1 \choose k-1}\ge  {n-k\choose k} + 2(k-2)$. This can be easily verified by using ${n-k\choose k} =\frac{n-2k+1}{k}{n-k\choose k-1} = \frac{k+t+1}{k}{n-k\choose k-1} < 2{n-k\choose k-1}$.
	
In the case when $t=k-2$, by Lemma \ref{le8} it suffices to show $\frac{1}{3}{n-1 \choose k-1} \ge \frac{1}{3}{n-k-1 \choose k}+\frac{2(k-2)}{3}+\frac{3}{4}$, which can be easily verified by using $n=4k-2 \ge 10$ and $\frac{1}{2}{n-1-k \choose k}=\frac{k-1}{k}{n-1-k \choose k-1}$. 
	 
If $t=k-1=2$, then $n=4k-1=11$ and by Lemma \ref{le8}, $h(\overline{K}(11,3)) \ge  60 > 55 = \frac{1}{3} {11 \choose 3}$. If $t=k-1=3$, then $n=4k-1=15$ and by Lemma \ref{le8}, $h(\overline{K}(15,4)) \ge  505 > 455= \frac{1}{3} {15 \choose 4}$.
	
Finally, in the case when $t=k-1\ge 4$, by Lemma \ref{le8} it suffices to show ${n-1 \choose k-1} + \frac{1}{(n-1)}{n-1 \choose k-1} \ge 2{n-2-k \choose k}+4(k-2)+9$, which can be verified by using $n=4k-1$ and $k\ge 5$.
  \end{proof}

\subsection{$s\ge 4$ and $k\ge 4$}

In this section we set
$$
l'= \left\lfloor (n-1)/(k-1) \right\rfloor,\;\;
l = \left\{ 
  \begin{array}{l l}
   \left\lfloor (l'+1)/2 \right\rfloor,   &  \text{if $(n,k)= (19,4)$}\\ [0.1cm]
   \left\lceil (l'+1)/2 \right\rceil,  &  \text{if $(n,k) \ne (19,4)$,}         
  \end{array} \right. \;\;
n' := n-l(k-1).
$$
Obviously,  $4\le s \le l'$ and $2 \le l < l'$.  
 
\begin{lemma}
\label{ob1}
With the notation above we have
\begin{itemize}  
\item[\rm (a)] $l \le \frac{l' + 2}{2} \le \frac{1}{2} \left(s + 3 + \frac{s-1}{k-1}\right)$;
\item[\rm (b)] $\frac{n}{2} < l(k-1) + 1 \le \frac{n-1}{2}+k$;  
\item[\rm (c)] $\frac{1}{l}{n-n' \choose k-1} > n'$.       
\end{itemize}   
\end{lemma}
 
\begin{proof}  
(a) The left-hand side inequality follows from the definition of $l$ and the right-hand side inequality follows from $l' = \left\lfloor \frac{s(k-1)+s+t-1}{k-1}\right\rfloor \le s + 1 + \frac{s-1}{k-1}$. 

(b) If $(n,k)\ne (19,4)$, then $l(k-1)+1\ge  \frac{1}{2}(l'+1)(k-1)+1 = l' \cdot \frac{k-1}{2} + \frac{k+1}{2}
	= \lfloor \frac{n-1}{k-1}\rfloor \cdot \frac{k-1}{2} +\frac{k+1}{2} \ge \frac{n-1-(k-2)}{k-1} \cdot \frac{k-1}{2} +\frac{k+1}{2}
	= \frac{n}{2}+1 > \frac{n}{2}$. If $(n,k)=(19,4)$, then
	 $l(k-1)+1=3 \cdot (4-1)+1> \frac{n}{2}$. On the other hand, since $l \le \lceil \frac{l'+1}{2}\rceil = \lceil \frac{1}{2} (\lfloor \frac{n-1}{k-1}\rfloor+1)\rceil \le   \frac{1}{2} (\lfloor \frac{n-1}{k-1}\rfloor+2) \le  \frac{1}{2} \cdot \frac{n-1}{k-1}+1$, we have $l(k-1) \le \frac{n-1}{2}+(k-1)$ no matter whether $(n,k)\ne (19,4)$ or not.
	 
(c) Since $l(k-1)+1 > \frac{n}{2}$ and $k\ge 4$, we have $\frac{1}{l}{n-n' \choose k-1} -n' = \frac{1}{l}{l(k-1) \choose k-1} -(n-l(k-1)) = {l(k-1)-1 \choose k-2} -(n-l(k-1)) > 0$.
 							\end{proof}
 
\begin{lemma}
\label{le2}
Let $n=sk+t$ be such that $s\ge 4$, $k \ge 4$ and $0\le t\le k-1$. Then 
$$
h(\overline{K}(n,k)) \ge \left\lceil {n \choose k}\Big/s \right\rceil =  \chi(\overline{K}(n,k)).
$$
\end{lemma}
  
\begin{proof}
Let $\AA_{i1}^l(n,k), \AA_{i2}^l(n,k), \ldots, \AA_{id_i}^l(n,k)$ be as in Lemma \ref{le3} each with size $l$, where $1 \le i \le n'$ and $d_i = \lfloor {n-i \choose k-1}/l \rfloor$. Then by Lemmas \ref{le3} and \ref{ob1}(b) the hyperedges of $\AA_{ij}^l(n,k)$ ($1 \le i \le n', 1\le j \le d_i$) cover at least $\min\{n-i+1, l(k-1)+1\}= l(k-1)+1 > \frac{n}{2}$ labels. Moreover, similar to the proofs of Lemmas \ref{le6} and \ref{le8}, for each $i$, $\cup_{j=1}^{d_i} \AA_{ij}^l(n,k)$ induces a complete subgraph of $\overline{K}(n,k)$. It follows that the subgraphs of $\overline{K}(n,k)$ induced by $\AA_{ij}^l(n,k)$, $1 \le i \le n'$, $1 \le j \le d_i$, give rise to a complete minor of $\overline{K}(n,k)$ with order $\sum_{i=1}^{n'} d_i$. It remains to prove $\sum_{i=1}^{n'} d_i \ge \frac{1}{s} {n \choose k}$. In fact, using Lemma \ref{ob1}(c), we have
\begin{eqnarray*}
\sum_{i=1}^{n'} d_i  & > & \frac{1}{l} \sum_{i=1}^{n'} {n-i \choose k-1} - n' \\ 
                   & = & \frac{1}{l} \sum_{i=1}^{n' - 1} {n-i \choose k-1} + \left(\frac{1}{l}{n-n' \choose k-1} -n'\right) \\  
                   & > & \frac{1}{l} \sum_{i=1}^{n' - 1} {n-i \choose k-1} \\  
                   & = & \frac{1}{l}\left({n \choose k} -{n-n'+1 \choose k}\right) \\  
                   & = & \frac{1}{l}(1 - f(n,k)){n \choose k},
\end{eqnarray*}  
where 
$$
f(n, k) := {n-n'+1 \choose k}\Big/{n \choose k} = {l(k-1)+1 \choose k}\Big/{n \choose k}.
$$ 
In what follows we prove $(1-f(n,k))s \ge l$ and thus complete the proof. 
	
Since $l(k-1)+1 \le \frac{n-1}{2}+k$ by Lemma \ref{ob1}(b), we have 		   
\begin{eqnarray}
f(n,k)  
& = & \prod_{j=0}^{k-1} \frac{(l(k-1)+1)-j}{n-j} \nonumber \\
& \le & \prod_{j=0}^{k-1} \frac{\frac{n-1}{2}+k-j}{n-j} \nonumber \\
& = & \prod_{j=0}^{k-1} \left(\frac{1}{2}+\frac{k-\frac{j+1}{2}}{n-j}\right). \lab{eq:fn}
\end{eqnarray} 
Denote the upper bound in (\ref{eq:fn}) by $g(n, k)$. Then $g(n, k) \le g(n-1, k)$ and it suffices to prove $(1-g(n,k))s \ge l$. 
 
\medskip
\textsf{Case 1: $k\ge 5$ and $s\ge 4$.}\, Note that $(k-1)/(sk-1) \le 1/s$ and $(k/2)/(sk-k+1) \le 1/(2s-2)$ for any $s \ge 1$ and $k \ge 1$, and $(k-\frac{j+1}{2})/(sk-j) < 1/s$ for any $s \ge 2, k \ge 1$ and $j \ge 0$. Hence
\begin{eqnarray}
\nonumber
g(n, k) & \le & g(sk, k) \nonumber \\ 
& = & \left(\frac{1}{2} + \frac{1}{s} - \frac{1}{2sk}\right) \cdot \left(\frac{1}{2} + \frac{k-1}{sk-1}\right) \cdot \prod_{j=2}^{k-2} \left(\frac{1}{2}+\frac{k-\frac{j+1}{2}}{n-j}\right) \cdot \left(\frac{1}{2} + \frac{\frac{k}{2}}{sk-k+1}\right) \nonumber \\
& < &  \left(\frac{1}{2} + \frac{1}{s}\right) \left(\frac{1}{2} + \frac{1}{s}\right) \left(\frac{1}{2} + \frac{1}{s}\right)^{k-3} \left(\frac{1}{2} + \frac{1}{2s-2}\right)  \nonumber \\
& = & \left(\frac{1}{2} + \frac{1}{s}\right)^{k-1} \left(\frac{1}{2} + \frac{1}{2s-2}\right)  \nonumber \\
& \le & \left\{ 
  \begin{array}{l l l}
    0.211,  & \quad \text{if $k\ge 5$ and $s=4$}\\ [0.1cm]
    0.151,  & \quad \text{if $k\ge 5$ and $s=5$}\\ [0.1cm]
    0.119, & \quad \text{if $k\ge 5$ and $s\ge6$}.
  \end{array} \right. \label{eq:gn4} 
\end{eqnarray} 
Since $k\ge 5$ and $l$ is an integer, by Lemma \ref{ob1}(a) we get $l\le s-1$ if $s=4$ or $5$. Thus, when $k\ge 5$ and $4\le s\le 5$, by (\ref{eq:gn4}) we have $(1-g(n,k))s \ge s-1 \ge l$ as required. Suppose that $k\ge 5$ and $s\ge6$.
Then $(1-g(n,k))s > (1-0.119)s = 0.881s$ by (\ref{eq:gn4}). Combining this with Lemma \ref{ob1}(a), it suffices to show $0.881s\ge \frac{1}{2} \left(s + 3 + \frac{s-1}{k-1}\right)$, that is, $0.762sk - 1.762s - 3k + 4 \ge 0$, which is satisfied as $k\ge 5$ and $s\ge 6$.

\medskip
\textsf{Case 2: $k = 4$ and $s \ge 4$.}\, We have 
\begin{eqnarray}
\nonumber
g(n, 4) & \le & g(4s, 4) \nonumber \\ 
& = & \prod_{j=0}^{3} \left(\frac{1}{2}+\frac{4-\frac{j+1}{2}}{4s-j}\right) \nonumber \\
& = &  \frac{4s+7}{8s}\cdot\frac{4s+5}{8s-2}\cdot\frac{4s+3}{8s-4}\cdot\frac{4s+1}{8s-6} \nonumber \\
& \le & \left\{ 
  \begin{array}{l l l l}
    0.224,  & \quad \text{if $s=4$}\\ [0.1cm]
    0.176,  & \quad \text{if $s=5$}\\ [0.1cm]
    0.149, & \quad \text{if $s=6$}\\ [0.1cm]
    0.133,   & \quad \text{if $s\ge 7$.}         
  \end{array} \right. \label{eq:gn5}
\end{eqnarray} 
It can be verified that for $4\le s\le6$ we have $l\le s-1$ (noting that $l=\lfloor (l'+1)/2 \rfloor$ when $(n,k)=(19,4)$) and hence $(1-g(n,4))s \ge s-1\ge l$ by (\ref{eq:gn5}) as required. If $s\ge 7$, then by (\ref{eq:gn5}), $(1-g(n, 4))s \ge 0.867s$. It can be verified that $0.867s \ge \frac{1}{2} (s + 3 + \frac{s-1}{3})$. This together with Lemma \ref{ob1}(a) implies $(1-g(n, 4))s \ge l$. 
\end{proof}

\subsection{$s\ge 4$ and $k=3$}

\begin{lemma}
\label{le12}
Let $n = 3s + t$ be such that $s \ge 4$ and $0 \le t \le 2$. If $n\ne 14$, then 
$$
h(\overline{K}(n,3)) \ge \left\lceil {n \choose 3}\Big/s \right\rceil = \chi (\overline{K}(n,3)).
$$
\end{lemma}
  
\begin{proof}  
Write  $n := 4s'+t'$, where $0\le t'\le 3$. Set	
  \[
  l = \left\{ 
  \begin{array}{l l}
    s',  & \quad \text{if $t'=0$ or $1$}\\
    s'+1,   & \quad \text{if $t'=2$ or $3$,}         
  \end{array} 
  \right. 
  \quad n'= n-2l.
  \]
Let $\AA_{i1}^l(n,3), \AA_{i2}^l(n,3), \ldots, \AA_{id_i}^l(n,3)$ be as in Lemma \ref{le3} each with size $l$, where $1 \le i \le n'$ and $d_i = \lfloor {n-i \choose 2}/l \rfloor$. Then by Lemma \ref{le3} hyperedges of $\AA_{ij}^l(n,3)$ ($1 \le i \le n', 1 \le j \le d_i$) cover at least $\min\{n-i+1, 2l+1\} = 2l + 1 > \frac{n}{2}$ labels.
Similar to the proof of Lemma \ref{le2}, one can verify that the subgraphs of $\overline{K}(n,3)$ induced by $\AA_{ij}^l(n,3)$, $1 \le i \le n'$, $1 \le j \le d_i$, give rise to a complete minor of $\overline{K}(n,3)$ of order $\sum_{i=1}^{n'} d_i$. It remains to prove $\sum_{i=1}^{n'} d_i \ge \frac{1}{s} {n \choose 3}$. Denoting $f(n) = \sum_{i=1}^{n'} d_i$, we have
\begin{eqnarray*}
 f(n)  & \ge & \frac{1}{l} \sum_{i=1}^{n'} {n-i \choose 2} - n' \\ 
                   & = & \frac{1}{l}\left({n \choose 3} -{n-n' \choose 3}\right) - n' \\  
                   & = & \frac{1}{l}\left({n \choose 3} -{2l \choose 3}\right) - (n-2l).  
                   \end{eqnarray*} 
Denote this lower bound by $g(n)$. One can verify that $\frac{n-1}{2} \le 2l \le \frac{n}{2} + 1$. Hence ${n \choose 3} -{2l \choose 3} \ge \frac{1}{6} \left(n(n-1)(n-2) - (\frac{n}{2}+1)(\frac{n}{2})(\frac{n}{2}-1)\right) = \frac{1}{48} n (7n^2 - 24n + 20) > 0$. Since  $l \le \frac{n+2}{4}$ and $\frac{n-2}{3} \le s \le \frac{n}{3}$, we have 
$g(n)-\frac{1}{s}{n \choose 3} 
\ge \frac{4}{n+2}\cdot\frac{1}{48} n (7n^2 - 24n + 20) - n + \frac{n-1}{2} -\frac{3}{n-2}{n \choose 3}
= \frac{n}{12(n+2)}(7n^2 - 24n + 20) -\frac{n^2+1}{2}
=\frac{1}{12(n+2)}(n^3-36n^2+14n-12)$. 
The function $x^3-36x^2+14x-12$ is monotonically increasing when $x \ge 24$, and it takes positive values when $x \ge 36$. Therefore, $f(n) \ge g(n) \ge \frac{1}{s}{n \choose 3}$ for $n \ge 36$ as required.  

In Table \ref{tab:hresult} we give the values of $\chi (n) = \lceil {n \choose 3}/s \rceil$ and at least one of $f(n)$ and $g(n)$ for $12 \le n \le 35$
with $n\notin\{14, 18, 22, 26\}$. Since $f(n) \ge g(n)$, we see from this table that for $12 \le n \le 35$ but $n \ne 14$, either $f(n) \ge \chi (n)$ as required or $f(n-1) \ge \chi (n)$. The latter case occurs when $n\in\{14, 18, 22, 26\}$, and in this case the subgraph $\overline{K}(n-1,3)$ of $\overline{K}(n,3)$ contains a complete minor of order at least $\chi (n)$. 
\end{proof}  
		  		
\begin{table}[h] 
\caption{Values of $f(n), g(n)$ and $\chi (n)$ for $12 \le n \le 35$} 
\smallskip

\centering 
\begin{tabular}{c rrrrrrrrrrrr}
\hline  
$n$    &  12& 13 & 14 & 15 & 16 &  17 & 18 &   19& 20 & 21   & 22& 23 \\ [0.5ex] 
\hline  
$l$      &  3 &3   & 4  & 4  &  4 &   4 &  5 &    5&   5&  5    & 6&  6 \\
$f(n)$      &    &   &    &    &  &     &    &  168 &  &      &  & 255\\
$g(n)$     & 60 & 81 &    & 92 & 118 &  147&  & & 194& 231  &  &   \\
$\chi (n)$ & 55 & 72 & 91 & 91 &112 &  136& 136&  162& 190& 190 & 220& 253 \\ [1ex] 
\hline 
\end{tabular}

\smallskip
\begin{tabular}{c rrrrrrrrrrrr}
\hline  
  $n$      & 24 & 25  &  26 & 27  &  28  &  29 & 30 & 31 & 32 & 33  & 34 & 35\\ [0.5ex] 
\hline  
$l$         & 6  &  6  &  7  & 7   &   7  &  7  & 8  & 8  &   8  &  8 &  9 &  9  \\
$f(n)$      &    &    &     &     &      &     &          \\
$g(n)$      & 288& 333 &     & 352 & 402 &  455& 423&476  &  534 & 595& 558& 619 \\
$\chi (n)$  & 253& 288 & 325 & 325 & 364  & 406 & 406&450 &496 & 496& 544& 595 \\ [1ex] 
\hline 
\end{tabular}
 
\label{tab:hresult}
\end{table}

\begin{lemma}
\label{le13} 
 $h(\overline{K}(14,3)) \ge \chi (\overline{K}(14,3))$.
\end{lemma}		
		
	\begin{proof}  
Since ${[14] \choose 3} \setminus \CC(14, 3)$ is simply ${[13] \choose 3}$, 
	by Table \ref{tab:hresult} and the proof of Lemma \ref{le12} we know that ${[14] \choose 3} \setminus \CC(14, 3)$ contains a complete minor of order at least $f(13) = 88$ such that the hyperedges (of $K_{14}^3$) in each of its branch sets cover at least $\lceil 13/2 \rceil = 7$ labels of $[13]$.
 
By Lemma \ref{le4}, $\CC(14, 3)$ can be partitioned into $\lceil {13 \choose 2}/4 \rceil =  20$ hypergraphs $\CC_{1}^4(14, 3)$, $\CC_2^4(14, 3)$, $\ldots$, $\CC_{20}^4(14, 3)$ with $|\CC_{i}^4(14, 3)| = 4$ ($1 \le i \le 19$) and $|\CC_{20}^4(14, 3)| = 2$ such that the hyperedges of $\CC_i^4(14, 3)$ ($1 \le i \le 19$) cover $\min\{14, 4 \cdot (3-1)+1\} = 9$ labels of $[14]$. Thus there is at least one edge between $\CC_i^4(14, 3)$ ($1 \le i \le 19$) and each branch set of the complete minor in the previous paragraph. On the other hand, each $\CC_i^4(14, 3)$ ($1 \le i \le 19$) induces a complete subgraph of $\overline{K}(14,3)$ since all its members contain label $14$. Therefore, $\overline{K}(14,3)$ has a complete minor of order at least $88 + 19  > 91 = \chi (\overline{K}(14,3))$.  
	\end{proof}
			    
Theorem \ref{th1} follows from Corollaries \ref{cor1} and \ref{cor2} and Lemmas \ref{le2}, \ref{le12} and \ref{le13}.  
        
\vskip 1pc 
{\bf Acknowledgements}~~  
The authors would like to thank the referees for their careful reading and valuable comments.

\end{document}